\documentclass[a4paper,12pt]{article}
\usepackage{amssymb,latexsym,amsmath,enumerate,verbatim,amsfonts,amsthm}
\usepackage[active]{srcltx}

\newtheorem{lemma}{Lemma}[section]

\newtheorem{theorem}{Theorem}[section]
\newtheorem{corollary}{Corollary}[section]
\newtheorem{proposition}{Proposition}[section]

\newtheorem{example}{Example}[section]

\begin{document}
\title{\Large{\sc{On Polynomial Optimization over Non-compact Semi-algebraic Sets}}
}

\author{V. Jeyakumar,\footnote{Department of Applied Mathematics, University of New South Wales,
Sydney 2052, Australia. E-mail: v.jeyakumar@unsw.edu.au} \quad J.B. Lasserre \footnote{LAAS-CNRS and Institute of Mathematics, LAAS, France, E-mail: lasserre@laas.fr. The work of this author was partially done while he was a Faculty of Science Visiting Fellow at UNSW}  \, and \, G. Li \footnote{Department of Applied Mathematics, University of New South Wales,
Sydney 2052, Australia. E-mail: g.li@unsw.edu.au} }

\date{\today}
\maketitle


\maketitle

\begin{abstract}
We consider the class of polynomial optimization problems
$\inf \{f(x):x\in K\}$ for which the quadratic module
generated by the polynomials that define $K$ and the polynomial $c-f$ (for some scalar $c$) is Archimedean.
For such problems, the optimal value can be approximated as closely as desired by solving a
hierarchy of semidefinite programs and the convergence is finite generically. Moreover, the Archimedean condition
(as well as a sufficient coercivity condition) can also
be checked numerically by solving a similar hierarchy of semidefinite programs.
In other words, under reasonable assumptions
the now standard hierarchy of SDP-relaxations extends to the non-compact case via a
suitable modification.
\medskip



\end{abstract}
\newpage
\section{Introduction}
The last decade has seen several developments in polynomial optimization. In
particular, a systematic procedure has been established to solve Polynomial Optimization
Problems (POP) on compact basic semi-algebraic sets. It consists of a
hierarchy of (convex) semidefinite relaxations of increasing size whose associated
sequence of optimal values is monotone nondecreasing and converges to the global
optimum. The proof of this convergence is based on powerful theorems from real
algebraic geometry on the representation of polynomials that are positive on a basic
semi-algebraic set, the so-called Positivstellensatze of Schm\"{u}dgen \cite{Schmudgen} and Putinar \cite{Putinar}.

Under mild assumptions the convergence has been proved to be finite for the class
of convex POPs and even at the first step of the hierarchy for the subclass of
convex POPs defined with SOS-convex polynomials\footnote{A polynomial $f$ is SOS-convex if its Hessian $\nabla^2 f(x)$ factors as $L(x)L(x)^T$ for some matrix
polynomial $L(x)$.}. In addition, as recently proved by Nie \cite{Nie1} and Marshall \cite{marshall},
finite convergence is {\it generic} for POPs on compact basic semi-algebraic sets.

However, all the above results hold in the compact case, i.e., when the feasible set
$K$ is a compact basic semi-algebraic set and (for the most practical hierarchy)
its defining polynomials satisfy an additional Archimedean assumption.
A notable exception is the case of SOS-convex POPs for which convergence is finite
even if $K$ is not compact (and of course if $f$ has a minimizer in $K$).

When the feasible set is a non compact basic semi-algebraic set, Schm\"udgen and Putinar's Positivstellensatz\"e
do not hold any more and in fact, as shown in Scheiderer \cite{Scheiderer}, there
are fundamental obstructions to such representations in the noncompact case.
The non compact case $K=\mathbb{R}^n$ reduces to the compact case if one guesses a
ball in which a minimizer exists or one may optimize over the gradient ideal
via the specialized hierarchy proposed in Nie et al. \cite{nie-gradient}. In both cases one assumes that a minimizer exists
which can be enforced if instead one minimizes an appropriate perturbation of the initial polynomial $f$
as proposed in Hanzon and Jibetean \cite{hanzon} and Jibetean and Laurent \cite{jibetean}.
To avoid assuming existence of a minimizer Schweighofer \cite{schweighofer} introduced
the notion of gradient tentacle along with
an appropriate hierarchy of SDP relaxations, later improved
by H\`a  and Vui \cite{tien-son} who instead use the truncated tangency variety. Remarkably, both hierarchies converge
to the global minimum even if there is no minimizer.

On the other hand the so-called Krivine-Stengle Positivstellensatz
provides a certificate of positivity even in the non compact case. Namely it states that a polynomial $f$ is positive on $K$ if and only if $pf = 1+q$ for some polynomials
$p$ and $q$ that both must belong to the preordering associated with the polynomials
that define $K$.
However, the latter representation is not practical for two reasons: Firstly, requiring that $p$ and $q$ belong to the preordering introduces $2^{m+1}$ unknown
SOS polynomials (as opposed to $m+1$ SOS polynomials in Putinar's Positivstellensatz
for the compact case). And so, for example, given a polynomial $f$, checking whether or not $pf = 1 + q$ for some polynomial $p, q$ in the preordering, is very costly from a computational viewpoint. Secondly, as the unknown polynomial $p$ multiplies $f$, this representation is not
practical for optimization purposes when $f$ is replaced with $f -\lambda$ where $\lambda$ has to be
maximized. Indeed one cannot define a hierarchy of semidefinite programs because
of the nonlinearities introduced by the product $p$.

\subsection*{Contribution}

In the present paper we consider the POP
\[\mathbf{P}:\qquad f^*\,=\,\inf \,\{f(x)\::\: x\,\in\,K\},\]
for a possibly non-compact basic semi-algebraic
set
$$K := \{x \in\mathbb{R}^n: g_j(x) \le 0; \:j = 1,\ldots, m; \quad h_l(x) = 0; \:l = 1\dots,r\},$$
for some polynomials $(f,g_j,h_l)\subset\mathbb{R}[x]$.
We assume that $\mathbf{P}$ is well-posed in the sense that
$f^*= f(x^*)$ for at least one minimizer $x \in K$.
A typical counter example is $f(x,y) = \inf\{x + (1-xy)^2 :
x \ge 0\}$ where $f^* = 0$ is not attained and $(\frac{1}{k},k)$ is an unbounded minimizing sequence.

An important class of well-posed POPs are those for which $\tilde{K}:=\{x\in K : f(x) \le c \}$
is compact, where $c \ge f(x_0)$ for some $x_0 \in K$. In particular notice that this holds
true with $c = f^*$ whenever $f$ has finitely many (global) minimizers in $K$ (the generic
case). \\

Our contribution is to show that if the quadratic module,
$$M(g; h; c-f)\:\left( := M(g_1,\ldots,g_m; h_1,\ldots,h_r; c -f)\right)$$
generated by the $g_j$'s, the $\pm h_l$'s and the polynomial $c - f$, is Archimedean, then
one may approximate as closely as desired the optimal value $f^*$ of $\mathbf{P}$.
Indeed it suffices to replace $K$ with $\tilde{K}:=K\cap\{x: c-f(x)\geq0\}$ and apply the
the associated standard hierarchy of semidefinite relaxations defined for the compact case.
That is, one solves the hierarchy of semidefinite programs:
\begin{equation}\label{H}
f_k = \sup  \{\lambda: f -\lambda \in  M_k(g; h; c-f) \},\qquad k\in\mathbb{N},
\end{equation}
where $M_k(g; h; f)$ is the restricted version of the quadratic module $M(g; h; c-f)$ in
which the polynomial weights have a degree bound that depend on $k$.
And if the quadratic module $M(g;h;c-f)$ is Archimedean
then the monotone convergence $f_k\uparrow f^*$ as $k\to\infty$, is guaranteed\footnote{The hierarchy of semidefinite programs (\ref{H}) was studied in \cite{Pham} for convex POPs for which the monotone convergence, $f_k\uparrow f^*$ as $k\to\infty$, has been proved to be true always (i.e. without any regularity condition).}
and, is finite generically. Moreover,
in such a context our result is to be interpreted as a simplified version of the celebrated Krivine-Stengle Positivstellensatz.

From a mathematical point of view this is a relatively straightforward extension as it
reduces the non compact case to the compact case by using $\tilde{K}$ instead of $K$.
However, one main goal of the paper
is to show that, under some numerically checkable assumptions,
the standard hierarchy of SDP relaxations defined for the compact case
indeed can be adapted to the non compact case modulo a slight modification.
For instance, when $f$ is coercive, the set $\{x \in \mathbb{R}^n : c -f(x) \ge 0\}$ is compact
(and so the Archimedean condition is satisfied). This coercivity condition of the objective function $f$ to minimize is very natural in
many POPs as it simply means that $f(x)$ grows to infinity as $\|x\|\to \infty$ (e.g. when $f$ is a strongly convex polynomial).

Importantly, we also show that the Archimedean and coercivity conditions can be
checked numerically by solving a hierarchy of semidefinite programs until
some test is passed.


\section{Preliminaries}

We write $f \in \mathbb{R}[x]$ (resp. $f \in \mathbb{R}_d[x]$) if $f$ is a real polynomial on $\mathbb{R}^n$ (resp. $f$ is a real polynomial with degree at most $d$).
One can associate the $l_1$-norm on $\mathbb{R}[x]$ defined by $\|f\|_1=\sum_{\alpha}|f_{\alpha}|$ where $f(x)=\sum_{\alpha}f_{\alpha}x^{\alpha}$.
We say that a real polynomial $f$ is sum of squares (SOS) (cf. \cite{Survey}) if there exist real polynomials $f_j$, $j=1,\ldots,r$, such that $f=\sum_{j=1}^rf_j^2$.
 The set of all sum of squares real polynomials is denoted by
 $\Sigma^2[x]$ while $\Sigma^2_d[x]\subset\Sigma^2[x]$ denotes its subset
  of all sum of squares of degree  at most $2d$.\\

\noindent
{\bf Coercivity.} We say that a polynomial $f$ is coercive if $\lim_{\|x\| \rightarrow +\infty}f(x)=+\infty$. Clearly, a coercive polynomial must be of even degree. A typical example of a coercive polynomial is that
 $f(x)=\sum_{i=1}^na_ix_i^{2d}+g(x)$ where $a_i>0$ and $g\in\mathbb{R}[x]_{2d-1}$.
This fact also implies that the set of coercive polynomials
is {\em dense} in $\mathbb{R}[x]$ for $l_1$-norm because, for any $f \in \mathbb{R}[x]$ with degree $p$, $f+\epsilon \sum_{i=1}^nx_i^{2p}$ is coercive.\\

\noindent
{\bf Archimedean property.}
With a semi-algebraic set $K$ defined as
\begin{equation}
\label{set-K}
K\,:=\,\{x\::\:g_j(x) \ge 0,\:j=1,\ldots,m;\quad  h_l(x)=0, \:l=1,\ldots,r\},
\end{equation}
is associated
its {\em quadratic module} $M(g;h)=M(g_1,\ldots g_m;h_1,\ldots,h_r)\subset\mathbb{R}[x]$ defined as
\begin{eqnarray*}
M(g;h) &:=& \{\sigma_0+ \sum_{j=1}^m\sigma_j(x)g_j(x)+\sum_{l=1}^r\phi(x)h_l(x) : \\
& &  \sigma_j\,\in\,\Sigma^2[x],\: j=0,1,\ldots,m;\quad \phi_l \in \mathbb{R}[x], \, l=1,\ldots,r\}.
\end{eqnarray*}
For practical computation we also have the useful truncated version $M_k(g;h)$ of $M(g;h)$, i.e.,
letting $v_j:=\lceil{\rm deg}(g_j)/2\rceil$, $j=1,\ldots,m$, and $w_l:={\rm deg}(h_l)$, $l=1,\ldots,r$,
\begin{eqnarray}
\label{useful}
M_k(g;h) &:=& \{\sigma_0+ \sum_{j=1}^m\sigma_j(x)g_j(x)+\sum_{l=1}^r\phi(x)h_l(x) : \\
\nonumber
& &  \sigma_j\,\in\,\Sigma^2[x]_{k-v_j},\: j=0,1,\ldots,m;\quad \phi_l \in \mathbb{R}[x]_{2k-w_l}, \, l=1,\ldots,r\}.
\end{eqnarray}

Of course, membership in $M(g;h)$ provides immediately with
a certificate of non negativity on $K$. The quadratic module $M(g;h)$
is called {\em Archimedean} if there exists a polynomial $p \in M(g;h)$ such that the superlevel set
$\{x \in \mathbb{R}^n: p(x)\ge 0\}$ is compact. An equivalent definition is that there exists $R>0$ such that
$x\mapsto R-\sum_{i=1}^nx_i^2 \in M(g;h)$. Observe that  if $M(g;h)$ is Archimedean then the set  $K$ is compact.

The following important result on the representation of polynomials that are positive on $K$ is from Putinar \cite{Putinar}.

\begin{theorem}{\bf (Putinar Positivstellensatz \cite{Putinar})}
\label{Putinar}
 Let $K\subset\mathbb{R}^n$ be as in (\ref{set-K}) and assume that $M(g;h)$ is Archimedean.
 If $f\in\mathbb{R}[x]$ is positive on $K$ then $f\in M(g;h)$.
 \end{theorem}
 An even more powerful result due to Krivine-Stengle is valid on more general semi-algebraic set (not necessarily compact).
Denote by $P(g;h)\subset\mathbb{R}[x]$ the {\em preordering} associated with $K$ in (\ref{set-K}), i.e., the set defined by
 \[P(g;h)\,:=\,\left\{\sum_{\alpha\in \{0,1\}^m} \sigma_\alpha \,g_1^{\alpha_1}\ldots g_m^{\alpha_m}+\sum_{l=1}^r \phi_l h_l: \sigma_\alpha \in \Sigma^2[x], \phi_l \in \mathbb{R}[x]\right\}.
\]
\begin{theorem}{\bf (Krivine-Stengle Positivstellensatz)}
\label{Krivine-Stengle}
Let $K\subseteq\mathbb{R}^n$ be as in (\ref{set-K}).

(i) If  $f\in\mathbb{R}[x]$ is nonnegative on $K$ then $pf=f^{2s}+q$ for some
$p,q\in P(g;h)$ and some integer $s$.

(ii) If  $f\in\mathbb{R}[x]$ is positive on $K$ then $pf=1+q$ for some
$p,q\in P(g;h)$.
\end{theorem}
Notice the difference between Putinar and Krivine-Stengle certificates.
On the one hand, the latter is valid for non-compact sets $K$ but requires knowledge of
two elements $p,q\in P(g;h)$, i.e., $2^{m+1}$ SOS polynomial weights
associated with the $g_j$'s and $2r$ polynomials associated with the $h_l$'s,
in their representation. On the other hand, the former is valid
only for compact sets $K$ with the Archimedean property, but it requires knowledge of only
$m+1$ SOS weights and $r$ polynomials.

\section{Main Results}
With $K$ as in (\ref{set-K}) and $f\in\mathbb{R}[x]$, let $d:=\max[{\rm deg}\,f;{\rm deg}\,g_j;{\rm deg}\,h_l]$ and let $s(d):={n+d\choose n}$.
We say that a property holds {\it generically} for $f,g_j,h_l\in\mathbb{R}[x]_d$
if the coefficients of the polynomials $f,g_j,h_l$ (as vectors of $\mathbb{R}^{s(d)}$)
do not satisfy a system of finitely many polynomial equations. Equivalently,
if the coefficients of $f,g_j,h_l$ belong to an open Zariski subset of
$\mathbb{R}^{(1+m+r)s(d)}$.\\

A consequence of Theorem \ref{Krivine-Stengle} is that, computing the global minimum of $f$
on $K$ reduce to solving the optimization problem
\begin{eqnarray}
\nonumber
f^*&=&\sup_\lambda\{\lambda\::\: \mbox{$f-\lambda >0$ on $K$}\}\\
\label{bad}
&=&\sup_{p,q,\lambda}\{\lambda\::\: p(f-\lambda)=1+q;\quad p,\,q\in P(g;h)\}.
\end{eqnarray}
But as already mentioned for Krivine-Stengle's Positivstellensatz, the above formulation (\ref{bad}) is
not appropriate because of the product $p\lambda$ which precludes from reducing
(\ref{bad}) to a semidefinite program. Moreover there are $2^{m+2}$ SOS polynomial weights in the definition of $p$
and $q$ in (\ref{bad}).
However, inspired by Theorem \ref{Krivine-Stengle} we now provide sufficient conditions on the data
$(f,g,h)$ of problem $\mathbf{P}$ to provide a converging hierarchy of semidefinite programs
for solving $\mathbf{P}$.

\begin{theorem}\label{th:2.2}
Let $\emptyset\neq K\subset\mathbb{R}^n$ be as in (\ref{set-K}) for some
polynomials $g_j,h_l$, $j=1,\ldots,m$ and $l=1,\ldots,r$.
Let $x_0 \in K$ and let $c \in \mathbb{R}_+$ be such that $c>f(x_0)$.
Suppose that the quadratic module $M(g;h;c-f)$ is
Archimedean. Then:
\begin{eqnarray}\label{eq:hierachy}
\nonumber
f^*&=&\inf_x \,\{f(x)\::\:x \in K\}\\
\label{good}
&=&\sup_{\lambda}\,\{\lambda\::\: f-\lambda\,\in\,M(g;h;c-f)\,\} \nonumber \\
& = & \lim_{k \rightarrow \infty} \sup_{\lambda}\{\lambda\::\:f-\lambda\,\in\,M_k(g;h;c-f)\}.
\end{eqnarray}
Moreover $f^*=f(x^*)$ for some $x^*\in K$, and generically
\begin{equation}
\label{finite-conv}
f^*=\max_{\lambda}\,\{\lambda\::\: f-\lambda\,\in\,M_k(g;h;c-f)\,\},
\end{equation}
for some index $k$. That is, $f^*$ is obtained after solving finitely many semidefinite programs.
\end{theorem}
\begin{proof}
It suffices to observe that
$f^*=\inf_x \{f(x):x \in K;\,c-f(x)\geq0\}$. And so if the quadratic module
$M(g;h;c-f)$ is Archimedean then the set $\tilde{K}:=K\cap \{x: c-f(x)\geq0\}$ is compact.
Therefore $f^*=f(x^*)$ for some $x^*\in\tilde{K}$. Moreover, $f_k\uparrow f^*$ as $k\to\infty$, where $f^*_k$ is the value of
the semidefinite relaxation (\cite{Lasserre2001,Lasserre2009})
\[f^*_k\,:=\,\sup\{\lambda\::\:f-\lambda\,\in\,M_k(g;h;c-f)\},\qquad k\in\mathbb{N}.\]
Finally, invoking Nie \cite{Nie1}, finite convergence takes place generically.
\end{proof}

On the other hand, in the case of a convex polynomial optimization problem $\mathbf{P}$, exploiting  the special structure of convex polynomials, the monotone convergence, $f_k\uparrow f^*$ as $k\to\infty$, has been proved in \cite{Pham} to be true always without any regularity condition.
\medskip

Note that in passing that if $f>0$ on $K$ then $f>0$ on $\tilde{K}=K\cap\{x:c-f(x)\geq0\}$, and so if
$M(g;h;c-f)$ is Archimedean then by Theorem \ref{Putinar},
\[f\,=\,\underbrace{\sigma_0+\sum_{j=1}^m\sigma_j\,g_j+\sum_{l=1}^r\phi_l\,h_l}_{q\in M(g;h)}+\psi(c-f),\]
for some $q\in M(g;h)$ and some SOS polynomial  $\psi\in\Sigma^2[x]$. Equivalently,
$(1+\psi)f=q+\underbrace{c\psi}_{SOS}$, i.e., $(1+\psi) f\in M(g;h)$ for some
SOS polynomial $\psi\in\Sigma^2[x]$.\\

In other words, one has shown:
\begin{corollary}
\label{coro-1}
Let $\emptyset\neq K\subset\mathbb{R}^n$ be as in (\ref{set-K}) for some
polynomials $g_j,h_l$, $j=1,\ldots,m$ and $l=1,\ldots,r$.
Let $x_0 \in K$ and let $c \in \mathbb{R}_+$ be such that $c>f(x_0)$.
Suppose that the quadratic module $M(g;h;c-f)$ is
Archimedean.

If $f>0$ on $K$ then $(1+\psi) f\in M(g;h)$ for some
SOS polynomial $\psi\in\Sigma^2[x]$.
And generically, if $f\geq0$ on $K$ then $(1+\psi) f\in M(g;h)$ for some
SOS polynomial $\psi\in\Sigma^2[x]$.
\end{corollary}

Thus Corollary \ref{coro-1}  can be regarded as a simplified form of Krivine-Stengle's Positivstellensatz
which holds whenever the quadratic module $M(g;h;c-f)$ is Archimedean.

It is also worth noting that the assumption
``the quadratic module $M(g;h;c-f)$ is Archimedean''
is weaker than the assumption ``the quadratic module $M(g;h)$ is Archimedean''
used in Putinar's Positivstellensatz. Indeed, obviously if
$M(g;h)$ is Archimedean then so is $M(g;h;c-f)$ whereas the converse is not true in general (see Example \ref{ex1}).


\subsection{Checking the Archimedean property}

We have seen that the quadratic module $M(g;h;c-f)$ is Archimedean if and only if
there exists $N>0$ such that the quadratic polynomial
$x\mapsto N-\|x\|^2$ belongs to $M(g;h;c-f)$.
This is equivalent to:
\[\inf\,\{\lambda\::\: \lambda-\|x\|^2\,\in\,M(g;h;c-f)\}\,<\,+\infty,\]
which, in turn, is equivalent to
\begin{equation}
\label{rho-k}
\rho_k\,:=\,\inf\,\{\lambda\::\: \lambda-\|x\|^2\,\in\,M_k(g;h;c-f)\}\,<\,+\infty
\end{equation}
for some $k\in\mathbb{N}$.

We note that for each fixed $k\in\mathbb{N}$, solving (\ref{rho-k})
reduces to solving a semi-definite program. And so
checking  whether the Archimedean property is satisfied
reduces to solving the hierarchy of semidefinite programs (\ref{rho-k}), $k\in\mathbb{N}$,
until $\rho_k<+\infty$ for some $k$.


\begin{example}
\label{ex1}
{\rm  Consider the following two-dimensional illustrative
example where $f(x_1,x_2)=x_1^2+1$, $g_1(x_1,x_2)=1-x_2^2$ and $g_2(x_1,x_2)=x_2^2-1/4$. Let $c=2$. The corresponding hierarchy $(\ref{rho-k})$ reads
\begin{eqnarray*}
\rho_k\,=\, \ \inf\{\lambda\::\: \lambda-\|x\|^2&=&\sigma_0+\sigma(2-f)+\sigma_1g_1+\sigma_2g_2 \\
& & \mbox{ for some }  \sigma, \sigma_0,\sigma_1,\sigma_2 \in \Sigma^2_k[x]\}.
\end{eqnarray*}
Using the following simple code
\begin{verbatim}
sdpvar x1 x2 lower;
f=x1^2+1;
g=[1-x2^2;x2^2-1/4;2-f];
h=x1^2+x2^2
[s1,c1]=polynomial([x1,x2],2)
[s2,c2]=polynomial([x1,x2],2)
[s3,c3]=polynomial([x1,x2],2)
F = [sos(lower-h-[s1 s2 s3]*g), sos(s1), sos(s2), sos(s3)];
solvesos(F,lower,[],[c1;c2;c3;lower]);
 \end{verbatim}
from the matlab toolbox Yamlip \cite{Yamlip1,Yamlip2}, we obtain
$\rho_2=2>0$, and an optimal solution
\begin{eqnarray*}
\sigma_0&=&1.325558711-1.3103 x_1^2-1.3408 x_2^2+0.4609x_1^4+0.3885x_1^2x_2^2+0.4761x_2^4\\
\sigma&=&0.3443415642+0.4609x_1^2+0.1947 x_2^2\\
\sigma_1&=&0.3300997245-(1.4964e-15)x_1+(1.4078e-15)x_2+0.1938 x_1^2+0.4761 x_2^2\\
&&+(9.0245e-15) x_1x_2\\
\sigma_2&=&(2.570263721e-11)+(1.0557e-16)x_1+(1.0549e-16)x_2+(2.6157e-11)x_1^2\\
&&+(2.5534e-11)x_2^2+(1.0244e-15) x_1x_2.
\end{eqnarray*}
So we conclude that $M(g;h;c-f)$ is Archimedean.
On the other hand, clearly, $f(x)>0$ for all $x \in K=\{x:g_i(x_1,x_2) \le 0, i=1,2\}=\{(x_1,x_2):x_1 \in \mathbb{R}, x_2 \in [-1,-1/2]\cup [1/2,1]\}$. Direct verification gives that $f$ is not coercive and $M(g_1,g_2)$ is not Archimedean (as $K$ is noncompact). Let $\bar x=(0,-1)$ and let $c=2>1=f(\bar x)$. We have
already shown that $M(g,h,c-f)$ is Archimedean.  Direct verification shows that
\[
x_1^2+1=\frac{1}{2}+ 2x_1^2(1-x_2^2)+2x_1^2(x_2^2-1/4)+ \frac{1}{2}[2-(x_1^2+1)].
\]
So, letting $\delta=\frac{1}{2}$, $\sigma_1(x)=2x_1^2$, $\sigma_2(x)=2x_1^2$ and $\sigma(x)=\frac{1}{2}$, we see that the following positivity certification holds
\[
f=\delta+\sigma_1g_1 +\sigma_2 g_2+\sigma(c-f).
\]
}\end{example}

So Example \ref{ex1} illustrates the case where even if $K$ is not compact and $f$ is not coercive,
still the quadratic module $M(g;h;c-f)$ is Archimedean.

We next provide an easily verifiable condition guaranteeing that the quadratic module
$M(g;h;c-f)$ is Archimedean in terms of coercivity of the functions involved.

\begin{proposition}\label{prop:1}
 Let $K$ be as in (\ref{set-K}), $x_0 \in K$ and let $c \in \mathbb{R}_+$ be such that $c>f(x_0)$.

Then, the quadratic module $M(g;h;c-f)$ is Archimedean if
there exist $\alpha_0,\lambda_j \ge 0$, $j=1,\ldots,m$, and $\mu_l \in \mathbb{R}$, $l=1,\ldots,p$, such that the polynomial $\alpha_0 f-\sum_{j=1}^m\lambda_j g_j-\sum_{l=1}^p \mu_l h_l$ is coercive.

In particular, $M(g;h;c-f)$ is Archimidean if $f$ is coercive\footnote{As shown in \cite{Pham}, for a convex polynomial $f$, the positive definiteness of the Hessian of $f$ at a single point guarantees coercivity of $f.$
}.
\end{proposition}
\begin{proof}
To see  that the quadratic module $M(g;h;c-f)$ is Archimedean, note that from its definition,
\[
p:=\alpha_0 (c-f)+\sum_{j=1}^m\lambda_j g_j+\sum_{l=1}^p \mu_l h_l  \in M(g;h;c-f).
\]
Now, $\{x:p(x)\ge 0\}=\{x:\alpha_0 f(x)-\sum_{j=1}^m\lambda_j g_j(x)-\sum_{l=1}^p \mu_l h_l(x) \le \alpha_0 c\}$ is nonempty (as $x_0 \in  \{x:p(x)\ge 0\}$) and compact (by our coercivity assumption).
This implies that the quadratic module $M(g;h;c-f)$ is Archimedean.

The particular case when $f$ is coercive follows from the general case with $\alpha_0=1$ and $\lambda_j,\mu_l=0$ for all $j,l$.
\end{proof}

We now show how the coercivity of a nonconvex polynomial can easily be checked by solving a sequence of semi-definite programming problems.

\subsection{Checking the coercivity property}
For a non convex polynomial $f\in\mathbb{R}[x]_d$ with $d$ even, decompose $f$ as the sum
\[
f=f_0+f_1+\ldots,+f_d ,
\]
where each $f_i$, $i=0,1,\ldots,d$, is an homogeneous polynomial of degree $i$. Let
$x\mapsto \theta(x):=\Vert x\Vert^2-1$ and let $M(\theta)$ be the quadratic module
\[M(\theta)\,:=\,\{\sigma+\phi \,\theta\::\: \sigma\in\Sigma^2[x];\quad \phi\in\mathbb{R}[x]\}.\]
\begin{lemma}\label{l1}
If there exists $\delta>0$ such that $f_d(x) \ge \delta \|x\|^d$, then
$f$ is coercive and in addition:
\begin{eqnarray}
\label{aux1}
f_d(x) \ge \delta \|x\|^d&\Leftrightarrow&0<\sup_{\|x\|=1}\,\{\mu: f_d(x) - \mu \ge 0\}\\
\label{aux2}
&\Leftrightarrow&0<\sup\,\{\mu: f_d - \mu\,\in\,M(\theta)\}.
\end{eqnarray}
\end{lemma}
\begin{proof} Assume that there exists $\delta>0$ such that $f_d(x) \ge \delta \|x\|^d$. To see that $f$ is coercive, suppose, on the contrary, that there exists $\{x_k\} \subseteq \mathbb{R}^n$ and $M>0$ such that $\|x_k\| \rightarrow \infty$ and $f(x_k) \le M$ for all $k \in \mathbb{N}$. This implies that
\begin{equation}\label{eq:contradiction}
\frac{f(x_k)}{\|x_k\|^d} \le \frac{M}{\|x_k\|^d} \rightarrow 0, \mbox{ as } k \rightarrow \infty.
\end{equation}
On the other hand, as each $f_i$ is a homogeneous function with degree $i$, for each $i=0,1,\ldots,d-1$ we have
\[
\frac{f_i(x_k)}{\|x_k\|^d}=f_i(\frac{x_k}{\|x_k\|}) \frac{1}{\|x_k\|^{d-i}} \rightarrow 0, \mbox{ as } k \rightarrow \infty.
\]
So, this together with the hypothesis gives us that, for sufficiently large $k$,
\begin{eqnarray*}
\frac{f(x_k)}{\|x_k\|^d}&=&\frac{f_0(x_k)}{\|x_k\|^d}+\frac{f_1(x_k)}{\|x_k\|^d}+\ldots +\frac{f_d(x_k)}{\|x_k\|^d} \ge  \frac{\delta}{2},
\end{eqnarray*}
which contradicts (\ref{eq:contradiction}). Hence, $f$ is coercive.

By homogeneity, the condition, $f_d(x)\geq\delta\Vert x\Vert^d$ for all $x\in\mathbb{R}^n$, is equivalent to the condition that
$f_d(x)\geq\delta$,  for all $x\in B:=\{x:\Vert x\Vert=1\}$. And so,
$0<\delta\leq\rho:=\sup_{x\in B}\{\mu\,:\,f_d(x)-\mu\geq0\}$. Conversely, if
$\rho>0$ then $f_d(x)\geq\rho \Vert x\Vert^d$ for all $x$ and so the equivalence (\ref{aux1}) follows.
Then the equivalence (\ref{aux2}) also follows  because $B$ is compact and
the quadratic module $M(\theta)$ is Archimedean.
\end{proof}

It easily follows from Lemma \ref{l1} that the sufficient coercivity condition that $f_d(\cdot) \ge \delta \|\cdot\|^d$ for some $\delta>0$ can be numerically checked by solving the following hierarchy of semidefinite programs:
\[\rho_k= \sup\{\mu: f_d-\mu \in M_k(\theta)\},\qquad k\in\mathbb{N},\]
until $\rho_k>0$ for some $k$.
The following simple example illustrates how to verify the coercivity of a polynomial by solving semidefinite programs.

\begin{example}
{\rm With $n=2$ consider the degree $6$ polynomial
\[x\mapsto f(x)\,:=\,x_1^6+x_2^6-x_1^3x_2^3+x_1^4-x_2+1.\]
To test whether $f$ is coercive, consider its highest degree term
 \[ x\mapsto f_6(x)\,=\,x_1^6+x_2^6-x_1^3x_2^3 \]
and the associated hierarchy of semidefinite programs:
\[\rho_k=\sup_{\mu,\phi,\sigma}\{\mu: f_6+\phi\, \theta - \mu=\sigma;\quad \phi \in \mathbb{R}_{2k}[x],\quad \sigma \in \Sigma^2_k[x]\},\quad k\in\mathbb{N}.\]
Running the following simple code
\begin{verbatim}
p=x^6+y^6-x^3*y^3;
g=[x^2+y^2-1]
[s1,c1]=polynomial([x,y],4)
F = [sos(p-lower-s1*g)];
solvesos(F,-lower,[],[c1;lower]);
 \end{verbatim}
from the SOS matlab toolbox Yamlip \cite{Yamlip1,Yamlip2}, one obtains
$\rho_4=0.125>0$, which proves that $f$ is coercive.
Indeed, from an optimal solution, one can directly check that
$x\mapsto f_6(x)-0.125(x_1^2+x_2^2)^3$ is an SOS polynomial of degree $6$.
So, $f_6(x) \ge 0.125(x_1^2+x_2^2)^3$, and
hence $f$ is coercive.
 \begin{verbatim}
u = p-0.125*(x^2+y^2)^3;
F = sos(u);
solvesos(F);
 \end{verbatim}
}\end{example}

\section{Conclusion}

In this paper we have first provided a simplified version of Krivine-Stengle's Positivstellensatz which holds generically.
The resulting positivity certificate is much simpler as it only involves an SOS polynomial
and an element of the quadratic module rather than two elements of the preordering.
And so it is also easier to check numerically by semidefinite programming. Inspired by this simplified form
we have also shown how to handle POPs on non compact basic semi-algebraic sets provided that
some quadratic module is Archimedean. The latter condition (or a sufficient coercivity condition)
can both be checked by solving a now standard hierarchy of semidefinite programs. The quadratic module is an easy and slight modification of the standard quadratic module when $K$ is compact, which shows that essentially
the non compact case reduces to the compact case when this Archimedean assumption is satisfied.
It is worth noting that the sufficient coercivity condition of the objective function $f$ to minimize is very natural in
many POPs as it simply means that $f(x)$ grows to infinity as $\|x\|\to \infty$.

\end{document}